\documentclass[11pt]{amsart}

\usepackage{geometry}
\usepackage{amssymb}
\usepackage{mathtools}
\usepackage{xcolor}
\usepackage[colorlinks=true,citecolor=red,linkcolor=blue,urlcolor=red]{hyperref}

\DeclareMathOperator{\Sym}{Sym}

\DeclareMathOperator{\Proj}{Proj}
\DeclareMathOperator{\Mor}{Mor}
\DeclareMathOperator{\Frac}{Frac}
\DeclareMathOperator{\trdeg}{trdeg}

\newcommand{\Mm}{\mathbf{M}}

\numberwithin{equation}{section}

\newtheorem{thm}{Theorem}[section]
\newtheorem{lem}[thm]{Lemma}
\newtheorem{prop}[thm]{Proposition}

\theoremstyle{definition}
\newtheorem{defn}[thm]{Definition}

\newtheorem{rem}[thm]{Remark}
\newtheorem{setup}[thm]{Set-up}

\begin{document}

\title{A klt generalized pair with infinitely generated canonical ring}

\author{Jihao Liu}
\address{Department of Mathematics, Peking University, No. 5 Yiheyuan Road, Haidian District, Beijing 100871, China}
\address{Beijing International Center for Mathematical Research, Peking University, No. 5 Yiheyuan Road, Haidian District, Beijing 100871, China}
\email{liujihao@math.pku.edu.cn}

\author{Yanze Wang}
\address{Academy of Mathematics and Systems Science, Chinese Academy of Sciences, No. 55 Zhongguancun East Road, Haidian District, Beijing, 100190, China}
\email{wangyanze@amss.ac.cn}

\date{\today}

\begin{abstract}
We construct a projective klt generalized pair over the complex numbers whose generalized log canonical ring is not finitely generated. The main result of this paper is obtained by generative AI, particularly GPT-5.6-sol-ultra, Fable 5, and the Danus system.
\end{abstract}

\maketitle

\section{Introduction}\label{sec:introduction}

By the celebrated finite generation theorem \cite[Corollary~1.1.2]{BCHM10}, the log canonical ring
\[
R(X,K_X+B):=\bigoplus_{m\geq 0}H^0(X,\mathcal O_X(\lfloor m(K_X+B)\rfloor))
\]
of a projective klt pair $(X,B)$ over $\mathbb C$ is a finitely generated $\mathbb C$-algebra. Generalized pairs, introduced by Birkar and Zhang \cite{BZ16}, enlarge the category of pairs by an auxiliary nef part $\Mm$, and have become a central tool in modern birational geometry. 

It is natural to ask whether the finite generation theorem persists in this larger category, that is, whether the generalized log canonical ring $R(X,K_X+B+\Mm_X)$ of a projective klt generalized pair $(X,B,\Mm)$ is always finitely generated. At first glance, this seems to be impossible, as most positivity results such as non-vanishing and abundance fail for generalized pairs. The finite generation of the canonical ring is also known to fail for lc generalized pairs \cite[Example 2.2]{LX23}. However, the construction of such an example seems more difficult than expected. A key reason is that all current examples for weird properties of generalized pairs support the finite generation property of the canonical ring. For example, if non-vanishing fails, then the canonical ring is trivial hence obviously finitely generated. The failure of the finite generation for the log canonical case relies on the fact that big and nef but not semi-ample divisors do not have finitely generated canonical rings, yet a big and nef generalized canonical class of a klt generalized pair is automatically semi-ample. Moreover, for divisors with small Kodaira dimension, their canonical rings are automatically finitely generated, which indicates that we have to find counterexamples in high dimensions.

The first author first learned of the finite generation question from Paolo Cascini in 2023 (who referred him to Yoshinori Gongyo), and later got the question from Gongyo directly at the Simons conference in 2024. The first reference addressing this question can be traced back to a research summary of Gongyo in a 2023 annual report \cite{Gon24}. The first author has discussed the question with at least Caucher Birkar, Osamu Fujino, Christopher D. Hacon, Junpeng Jiao, Vladimir Lazić, and Lingyao Xie in the past few years; despite the general feeling that a negative answer is very likely, no precise counterexample could be found. Thanks to the development of generative AI, in this paper, we construct the following example which shows that the finite generation of the canonical ring fails for klt generalized pairs in general. It is worth mentioning that Gongyo and Takayama \cite{GT24} proved the finite generation of the canonical ring for klt generalized pairs in many special cases.

\begin{thm}\label{thm:main}
There exists a projective klt generalized pair $(X,B,\Mm)$ over $\mathbb C$ with the following properties.
\begin{enumerate}
\item $X$ is smooth, $B=0$, $\Mm$ descends on $X$, and there is a nef Cartier divisor $\xi$ on $X$ such that
$$K_X\sim -16\xi,\quad \Mm_X\sim 17\xi.$$
In particular, $-K_X$ is nef.
\item The generalized log canonical ring
    \[
    R(X,K_X+B+\Mm_X)=R(X,\xi)=\bigoplus_{m\geq 0}H^0(X,\mathcal O_X(m\xi))
    \]
    is not a finitely generated $\mathbb C$-algebra. In particular,
    $$R(X,-K_X)=R(X,16\xi)$$
    is not a finitely generated $\mathbb C$-algebra.
    \item We have
    $$\kappa(X,\xi)=11<15=\nu(X,\xi).$$
    In particular, $K_X+B+\Mm_X$ is nef but neither big nor abundant.
\end{enumerate}
\end{thm}

\begin{rem}
It seems that Theorem \ref{thm:main} also provides the first example such that $X$ is klt and $-K_X$ is nef but $R(X,-K_X)$ is not finitely generated. Indeed, in Theorem \ref{thm:main}, $X$ is smooth.
\end{rem}

Theorem~\ref{thm:main} delimits the known positive results. When $K_X+B+\Mm_X$ is big, finite generation follows from \cite{BCHM10} by perturbing the nef part into the boundary; Gongyo and Takayama prove finite generation for projective klt generalized pairs, with $K_X+B+\Mm_X$ of finite Cartier index, of dimension at most $3$ \cite[Theorem~5.6]{GT24}, as well as in arbitrary dimension under the analytic hypothesis that the nef part carries a semi-positive singular Hermitian metric with vanishing Lelong numbers \cite[Theorem~1.3]{GT24}. Theorem~\ref{thm:main} shows that neither the bigness assumption nor the analytic hypothesis can be removed in general, and that the dimension restriction cannot be extended to all dimensions. Part (3) shows that the example is not an artifact of small Iitaka dimension: the ring $R(X,K_X+B+\Mm_X)$ has transcendence degree $12$.

We now explain the construction. Totaro, refining work of Nagata \cite{Nag59} and Mukai \cite{Muk01} on Hilbert's fourteenth problem, constructed an explicit linear representation $V=\mathbb A^{16}$ of the vector group $U=(\mathbb G_a)^4$ over any field $k$ of characteristic not $2$ whose ring of invariants $\mathcal O(V)^U$ is not finitely generated \cite[Corollary~7.1]{Tot08}. On the other hand, by a theorem of Brion, the universal vector extension $G=E(A)$ of an abelian variety $A$ over a field of characteristic $0$ is \emph{anti-affine}: every global regular function on $G$ is constant \cite[Proposition~2.3(i)]{Bri09}. Taking $\dim A=4$, the group $G$ is a torsor under $U$ over $A$, and Totaro's representation $W=V^\vee$ gives rise to an associated vector bundle $\mathcal E=G\times^U W$ of rank $16$ on $A$, an iterated extension of trivial line bundles. Anti-affineness of $G$ forces every equivariant morphism from $G$ to a representation to be constant, and a descent argument then identifies
\[
\bigoplus_{m\geq 0}H^0(A,\Sym^m\mathcal E)\cong\bigoplus_{m\geq 0}(\Sym^mW)^U=\mathcal O(V)^U
\]
as graded $\mathbb C$-algebras. Now let $X=\mathbb P_A(\mathcal E)$ with tautological class $\xi$. Since $\det\mathcal E\cong\mathcal O_A$ and $K_A\sim 0$, the projective bundle formula gives $K_X\sim-16\xi$, so the choice $B=0$ and $\Mm=\overline{17\xi}$ yields a klt generalized pair with $K_X+B+\Mm_X\sim\xi$ and
\[
R(X,\xi)=\bigoplus_{m\geq 0}H^0(A,\Sym^m\mathcal E)\cong\mathcal O(V)^U,
\]
which is not finitely generated.

\begin{rem}\label{rem:dimension}
By \cite[Theorem~5.6]{GT24}, any example as in Theorem~\ref{thm:main} has dimension at least $4$, while our example has dimension $19$. It would be interesting to determine the minimal dimension of a projective klt generalized pair with infinitely generated canonical ring.
\end{rem}

\begin{rem}\label{rem:lelong}
By \cite[Theorem~1.3]{GT24} applied in contrapositive form, for the generalized pair of Theorem~\ref{thm:main} the nef Cartier divisor $\Mm_X=17\xi$, as well as its pullback to any smooth higher birational model of $X$, carries no semi-positive singular Hermitian metric whose local potentials have vanishing Lelong numbers at every point. This gives an algebraic construction of nef line bundles that are maximally far from semi-positive in the analytic sense, in the spirit of the examples of Demailly--Peternell--Schneider \cite{DPS94} on projectivized rank-two bundles over elliptic curves.
\end{rem}

\begin{rem}
The main result of this paper was obtained using generative AI, particularly GPT-5.6-sol-ultra, Fable 5, and the Danus system. Danus is a specialized agent built on the Rethlas system and substantially more capable of conducting fundamental mathematical research. Human verification and polishing were done afterwards. See \cite{Liu+26} and \cite{Ju+26} for detailed introductions to the Danus system and the Rethlas system, respectively. Due to the limitation of generative AI, it is possible that we have missed some related references in the literature, and we welcome any comments from experts.
\end{rem}

\subsection*{Acknowledgements}
The work was partially supported by the National Key R\&D Program of China \#\allowbreak 2024YFA1014400.
The first author would like to thank other members of the Danus team (namely Guoxiong Gao, Zeming Sun, Bin Wu, Shurui Liu, Jiedong Jiang, Haocheng Ju, Leheng Chen, Ronnie Cheng, Xiping Zhang, and Bin Dong) and the Rethlas team (namely Haocheng Ju, Jiedong Jiang, Shurui Liu, Guoxiong Gao, Yuefeng Wang, Zeming Sun, Bin Wu, Liang Xiao, and Bin Dong) for their contributions to the development of Danus and Rethlas.
The first author would like to thank Ruochuan Liu and Gang Tian for constant support and encouragement. The second author would like to thank Yifei Chen for constant support.

\section{Preliminaries}\label{sec:preliminaries}

We adopt the standard notation and terminology for the minimal model program from \cite{Sho92,KM98,BCHM10} and use them freely. We work over the field of complex numbers $\mathbb C$.

\subsection{Generalized pairs}

We recall the definitions we need from \cite{BZ16}, in the special form in which the nef part descends on $X$ itself; this is the only case used in this paper.

\begin{defn}[b-divisors and closures]\label{defn:bdivisor}
Let $X$ be a normal projective variety. A \emph{b-divisor} $\Mm$ on $X$ is a collection of $\mathbb R$-divisors $\Mm_Y$, one for each projective birational morphism $g\colon Y\to X$ from a normal variety $Y$, which is compatible under pushforward: if $h\colon Y'\to Y$ is a further projective birational morphism, then $h_*\Mm_{Y'}=\Mm_Y$. The divisor $\Mm_Y$ is called the \emph{trace} of $\Mm$ on $Y$. For a nef $\mathbb Q$-Cartier $\mathbb Q$-divisor $N$ on a projective birational model $X'\to X$, the \emph{closure} $\overline N$ is the b-divisor whose trace on every projective birational model $Y\to X$ dominating $X'$ is the pullback of $N$, and whose trace on an arbitrary model is the pushforward of this pullback from a common resolution. A b-divisor $\Mm$ is \emph{b-nef} and \emph{b-$\mathbb Q$-Cartier} if $\Mm=\overline{\Mm_{X'}}$ for such a pair $(X',\Mm_{X'})$; when one may take $X'=X$, we say that $\Mm$ \emph{descends on $X$}, and then the trace of $\Mm$ on every model $g\colon Y\to X$ is $g^*\Mm_X$. Only the descend-on-$X$ case is used in this paper.
\end{defn}

\begin{defn}[generalized pairs]\label{defn:gpair}
A \emph{generalized pair} $(X,B,\Mm)$ over $\mathbb C$ consists of a normal projective variety $X$, an effective $\mathbb Q$-divisor $B$ on $X$, and a b-nef b-$\mathbb Q$-Cartier b-divisor $\Mm$ on $X$, such that $K_X+B+\Mm_X$ is $\mathbb Q$-Cartier, where $\Mm_X$ is the trace of $\Mm$ on $X$.

Let $(X,B,\Mm)$ be a generalized pair and let $g\colon Y\to X$ be a log resolution of $(X,B)$ on which $\Mm$ descends. Write
\[
K_Y+B_Y+\Mm_Y=g^*(K_X+B+\Mm_X).
\]
We say that $(X,B,\Mm)$ is \emph{generalized klt} if every coefficient of $B_Y$ is $<1$. This condition is independent of the choice of $g$ \cite[Section~4]{BZ16}.

The \emph{generalized log canonical ring} of $(X,B,\Mm)$ is
\[
R(X,K_X+B+\Mm_X):=\bigoplus_{m\geq 0}H^0(X,\mathcal O_X(\lfloor m(K_X+B+\Mm_X)\rfloor)).
\]
\end{defn}

\begin{lem}\label{lem:linear-equivalence}
Let $X$ be a normal projective variety and let $D$ and $D'$ be $\mathbb Q$-Cartier $\mathbb Q$-divisors on $X$ with $D\sim D'$. Then $R(X,D)\cong R(X,D')$ as graded $\mathbb C$-algebras.
\end{lem}

\begin{proof}
Let $s$ be a nonzero rational function on $X$ with $D'=D+\operatorname{div}(s)$. Multiplication by $s^m$ defines an isomorphism $H^0(X,\mathcal O_X(\lfloor mD'\rfloor))\to H^0(X,\mathcal O_X(\lfloor mD\rfloor))$ for every $m\geq 0$, and these isomorphisms are compatible with the multiplication maps on both sides. Taking the direct sum over $m$, we obtain the desired isomorphism of graded $\mathbb C$-algebras.
\end{proof}

\subsection{Projective bundles}

We fix once and for all the following convention: for a locally free sheaf $\mathcal E$ of rank $r$ on a smooth projective variety $A$, we let
\[
\mathbb P_A(\mathcal E):=\Proj_A(\Sym\mathcal E),
\]
with projection $\pi\colon\mathbb P_A(\mathcal E)\to A$ and Serre twisting line bundle $\mathcal O_X(1)$, where $X=\mathbb P_A(\mathcal E)$, and a Cartier divisor $\xi$ on $X$ with $\mathcal O_X(\xi)\cong\mathcal O_X(1)$; all statements below are invariant under this choice by Lemma~\ref{lem:linear-equivalence}. Under this convention,
\begin{equation}\label{eq:pushforward}
\pi_*\mathcal O_X(m)=\Sym^m\mathcal E\quad\text{for every integer } m\geq 0
\end{equation}
by \cite[II, Proposition~7.11]{Har77}, and the relative Euler sequence
\begin{equation}\label{eq:euler}
0\to\Omega_{X/A}\to(\pi^*\mathcal E)\otimes\mathcal O_X(-1)\to\mathcal O_X\to 0
\end{equation}
gives
\begin{equation}\label{eq:relative-canonical}
\omega_{X/A}\cong\mathcal O_X(-r)\otimes\pi^*\det\mathcal E
\end{equation}
by \cite[III, Exercise~8.4]{Har77}.

A vector bundle $\mathcal E$ on $A$ is \emph{nef} if $\mathcal O_{\mathbb P_A(\mathcal E)}(1)$ is nef, and \emph{numerically flat} if both $\mathcal E$ and $\mathcal E^\vee$ are nef; this is the convention of \cite[Chapter~6]{Laz04}, whose projectivization is the same $\Proj_A(\Sym\mathcal E)$ as above. We use the following fact: an extension of nef vector bundles is nef \cite[Theorem~6.2.12]{Laz04}.

\subsection{Totaro's representation}

We use the following theorem of Totaro, which realizes Mukai's counterexamples to Hilbert's fourteenth problem by explicit representations.

\begin{thm}[{\cite[Corollary~7.1]{Tot08}}]\label{thm:totaro}
Let $k$ be a field of characteristic not $2$. Let $(\mathbb G_a)^8$ act on $V=\mathbb A^{16}$, with coordinate functions $x_1,\dots,x_8,y_1,\dots,y_8$, by Nagata's formula \cite{Nag59}
\[
(t_1,\dots,t_8)\cdot(x_1,\dots,x_8,y_1,\dots,y_8)=(x_1,\dots,x_8,y_1+t_1x_1,\dots,y_8+t_8x_8),
\]
and restrict the action to the subgroup $U\cong(\mathbb G_a)^4$ of $(\mathbb G_a)^8$ spanned by the rows of the matrix
\[
\begin{pmatrix}
0&0&0&0&1&1&1&1\\
0&0&1&1&0&0&1&1\\
0&1&0&1&0&1&0&1\\
1&1&1&1&1&1&1&1
\end{pmatrix},
\]
so that $u=(a,b,c,d)\in U$ acts through
\[
(t_1,\dots,t_8)=(d,\ c+d,\ b+d,\ b+c+d,\ a+d,\ a+c+d,\ a+b+d,\ a+b+c+d).
\]
Then the ring of invariants $\mathcal O(V)^U$ is not finitely generated over $k$.
\end{thm}

Geometrically, the eight vectors of coefficients above are the eight points $(x,y,z)\in\{0,1\}^3$ of a cube in $\mathbb A^3\subset\mathbb P^3$, which form the base locus of the net of quadrics $\langle x^2-xw,\,y^2-yw,\,z^2-zw\rangle$; the resulting elliptic fibration of the blow-up of $\mathbb P^3$ at these eight points has Mordell--Weil rank one, which produces infinitely many $(-1)$-divisors and forces infinite generation \cite[Theorem~7.2]{Tot08}.

Throughout the paper we write $W:=V^\vee$ for the dual representation, so that $\mathcal O(V)=\Sym W$ and
\begin{equation}\label{eq:invariant-ring}
\mathcal O(V)^U=\bigoplus_{m\geq 0}(\Sym^mW)^U
\end{equation}
as graded $\mathbb C$-algebras. Explicitly, $W$ has basis $x_1,\dots,x_8,y_1,\dots,y_8$, and $u\in U$ acts on $W$ by
\begin{equation}\label{eq:dual-action}
u\cdot x_i=x_i,\qquad u\cdot y_i=y_i-t_ix_i\qquad(1\leq i\leq 8),
\end{equation}
where $(t_1,\dots,t_8)$ is associated with $u$ as in Theorem~\ref{thm:totaro}.

\subsection{Anti-affine groups and the universal vector extension}

An algebraic group $G$ over a field $k$ is \emph{anti-affine} if $\mathcal O(G)=k$, that is, every global regular function on $G$ is constant. Every abelian variety $A$ over $\mathbb C$ admits a \emph{universal vector extension}
\begin{equation}\label{eq:uve}
0\to H^1(A,\mathcal O_A)^\vee\to E(A)\to A\to 0,
\end{equation}
an extension of $A$ by the vector group $H^1(A,\mathcal O_A)^\vee$ which is universal among extensions of $A$ by vector groups; see \cite[Section~2.2]{Bri09} and the references given there. We use the following theorem of Brion.

\begin{thm}[{\cite[Proposition~2.3(i)]{Bri09}}]\label{thm:brion}
Let $A$ be an abelian variety over a field of characteristic $0$. Then $E(A)$ is anti-affine.
\end{thm}

\section{The associated bundle and the descent lemma}\label{sec:descent}

In this section, we construct the vector bundle $\mathcal E$ and prove the descent lemma identifying $\bigoplus_m H^0(A,\Sym^m\mathcal E)$ with the invariant ring \eqref{eq:invariant-ring}.

Throughout this section, $A$ is an abelian variety over $\mathbb C$ of dimension $4$, $U:=H^1(A,\mathcal O_A)^\vee$ regarded as a vector group, and $q\colon G:=E(A)\to A$ is the universal vector extension \eqref{eq:uve}.

\begin{lem}\label{lem:torsor}
With the notation above:
\begin{enumerate}
    \item $U\cong(\mathbb G_a)^4$ as algebraic groups over $\mathbb C$;
    \item $q\colon G\to A$ is a $U$-torsor for the right translation action of $U$ on $G$, and this torsor is Zariski-locally trivial;
    \item $\mathcal O(G)=\mathbb C$.
\end{enumerate}
\end{lem}

\begin{proof}
(1) Since $A$ is an abelian variety of dimension $4$, the vector space $H^1(A,\mathcal O_A)$ has dimension $4$ over $\mathbb C$, and a vector group of dimension $4$ over $\mathbb C$ is isomorphic to $(\mathbb G_a)^4$.

(2) The morphism $q$ is a surjective smooth homomorphism of algebraic groups with kernel $U$. Let $U$ act on $G$ by right translation, $g\cdot u:=gu$. Since $q(u)$ is the identity element of $A$ for $u\in U$, this action preserves the fibers of $q$. If $g_1,g_2$ lie in the same fiber, then $q(g_1^{-1}g_2)$ is the identity, so $g_1^{-1}g_2\in U$ and $g_2=g_1\cdot(g_1^{-1}g_2)$; if $g\cdot u=g$, then $u$ is the identity. Thus the action is free and transitive on each fiber, and since $q$ is faithfully flat, $q\colon G\to A$ is a $U$-torsor. For Zariski-local triviality, note that $\mathbb G_a$ is special in the sense of Serre, and extensions and products of special groups are special \cite[Section~4]{Ser58}; hence $U\cong(\mathbb G_a)^4$ is special, and every $U$-torsor over a variety is Zariski-locally trivial.

(3) This is Theorem~\ref{thm:brion} applied over $\mathbb C$.
\end{proof}

\begin{setup}\label{setup:associated-bundle}
Let $L$ be a finite-dimensional algebraic representation of $U$ over $\mathbb C$. We define the \emph{associated bundle}
\[
F_L:=G\times^UL:=(G\times L)/U,
\]
where $U$ acts on $G\times L$ by
\begin{equation}\label{eq:pinned-action}
u\cdot(g,l):=(g\cdot u^{-1},\,u\cdot l).
\end{equation}
By Lemma~\ref{lem:torsor}(2) we may choose a Zariski open cover $\{A_i\}$ of $A$ and sections $s_i\colon A_i\to G$ of $q$. Every point of $q^{-1}(A_i)$ is uniquely of the form $s_i(a)\cdot u$ with $a\in A_i$ and $u\in U$, and the class of $(s_i(a)\cdot u,\,l)$ in $F_L$ equals the class of $(s_i(a),\,u\cdot l)$. Thus $F_L$ is identified over $A_i$ with $A_i\times L$; the transition functions over $A_i\cap A_j$ are given by the action of $u_{ij}\colon A_i\cap A_j\to U$, $s_j(a)=s_i(a)\cdot u_{ij}(a)$, through the representation $L$. In particular, $F_L$ is an algebraic vector bundle on $A$ of rank $\dim_{\mathbb C}L$, and for representations $L,L'$ every $U$-equivariant linear map $L\to L'$ induces a morphism of vector bundles $F_L\to F_{L'}$.

We write $\mathcal E:=F_W=G\times^UW$, where $W$ is the $16$-dimensional representation \eqref{eq:dual-action} of $U\cong(\mathbb G_a)^4$; here we fix once and for all an isomorphism $U\cong(\mathbb G_a)^4$ as in Lemma~\ref{lem:torsor}(1) and regard $W$ as a representation of $U$ through it.
\end{setup}

\begin{prop}\label{prop:bundle}
The associated bundle $\mathcal E=G\times^UW$ of Set-up~\textup{\ref{setup:associated-bundle}} is an algebraic vector bundle of rank $16$ on $A$ which admits a filtration
\[
0=\mathcal E_0\subset\mathcal E_1\subset\dots\subset\mathcal E_{16}=\mathcal E
\]
by subbundles with $\mathcal E_i/\mathcal E_{i-1}\cong\mathcal O_A$ for $1\leq i\leq 16$. Consequently:
\begin{enumerate}
    \item $\det\mathcal E\cong\mathcal O_A$ and $c_i(\mathcal E)=0$ for every $i\geq 1$;
    \item $\mathcal E$ is numerically flat; in particular, $\mathcal E$ is nef.
\end{enumerate}
\end{prop}

\begin{proof}
That $\mathcal E$ is a rank-$16$ vector bundle is part of Set-up~\ref{setup:associated-bundle}. By \eqref{eq:dual-action}, the subspace $W':=\operatorname{span}(x_1,\dots,x_8)\subset W$ is fixed pointwise by $U$, and $U$ acts trivially on the quotient $W/W'$, since $u\cdot y_i\equiv y_i \pmod{W'}$. Refine the two-step filtration $0\subset W'\subset W$ to a full flag
\[
0=W_0\subset W_1\subset\dots\subset W_{16}=W
\]
with $\dim W_i=i$ and $W_8=W'$. Each $W_i$ is $U$-stable, and $U$ acts trivially on each one-dimensional quotient $W_i/W_{i-1}$.

Set $\mathcal E_i:=G\times^UW_i$. In the local trivializations of Set-up~\ref{setup:associated-bundle}, $\mathcal E_i$ is identified with $A_i\times W_i$, so $\mathcal E_i$ is a subbundle of $\mathcal E$, and $\mathcal E_i/\mathcal E_{i-1}\cong G\times^U(W_i/W_{i-1})$. Since $U$ acts trivially on $W_i/W_{i-1}$, the transition functions of this quotient bundle are trivial, and $\mathcal E_i/\mathcal E_{i-1}\cong\mathcal O_A$. This proves the existence of the filtration.

(1) Taking determinants along the filtration, we obtain $\det\mathcal E\cong\bigotimes_{i=1}^{16}\det(\mathcal E_i/\mathcal E_{i-1})\cong\mathcal O_A$. By the Whitney sum formula applied along the filtration, the total Chern class of $\mathcal E$ is $c(\mathcal E)=\prod_{i=1}^{16}c(\mathcal O_A)=1$, so $c_i(\mathcal E)=0$ for $i\geq 1$.

(2) The line bundle $\mathcal O_A$ is nef. Applying \cite[Theorem~6.2.12]{Laz04} inductively to the exact sequences $0\to\mathcal E_{i-1}\to\mathcal E_i\to\mathcal O_A\to 0$, we obtain that $\mathcal E$ is nef. Dualizing each of these sequences gives $0\to\mathcal O_A\to\mathcal E_i^\vee\to\mathcal E_{i-1}^\vee\to 0$, so $\mathcal E^\vee$ is also an iterated extension of copies of $\mathcal O_A$, and the same induction shows that $\mathcal E^\vee$ is nef. Thus $\mathcal E$ is numerically flat.
\end{proof}

We now prove the key descent lemma. It is here, and only here, that anti-affineness of $G$ is used.

\begin{prop}[Descent lemma]\label{prop:descent}
Let $L$ be a finite-dimensional algebraic representation of $U$ over $\mathbb C$ and let $F_L$ be the associated bundle of Set-up~\textup{\ref{setup:associated-bundle}}. Then:
\begin{enumerate}
    \item $H^0(A,F_L)$ is naturally identified with the set $\Mor^U(G,L)$ of morphisms $f\colon G\to L$ of varieties satisfying $f(g\cdot u)=u^{-1}\cdot f(g)$ for all points $g$ of $G$ and $u$ of $U$;
    \item every morphism from $G$ to an affine space is constant, and consequently
    \[
    H^0(A,F_L)\cong L^U:=\{l\in L\mid u\cdot l=l \text{ for all } u\in U\};
    \]
    \item $\Sym^m\mathcal E\cong F_{\Sym^mW}$ for every $m\geq 0$, where $\Sym^mW$ carries the induced representation;
    \item the resulting isomorphisms $H^0(A,\Sym^m\mathcal E)\cong(\Sym^mW)^U$ are compatible with the multiplication maps $\Sym^a\otimes\Sym^b\to\Sym^{a+b}$ on both sides. In particular, there is an isomorphism of graded $\mathbb C$-algebras
    \[
    \bigoplus_{m\geq 0}H^0(A,\Sym^m\mathcal E)\cong\bigoplus_{m\geq 0}(\Sym^mW)^U=\mathcal O(V)^U.
    \]
\end{enumerate}
\end{prop}

\begin{proof}
(1) Since $q$ is a $U$-torsor, the pullback $q^*F_L$ is canonically trivialized: over a point $g$ of $G$, every element of the fiber $(F_L)_{q(g)}$ is the class of $(g,l)$ for a unique $l\in L$, because $U$ acts freely and transitively on the fiber of $q$ through $g$. Let $s\in H^0(A,F_L)$. Composing $q^*s$ with this trivialization yields a morphism $f_s\colon G\to L$, characterized by
\[
s(q(g))=[(g,f_s(g))]\quad\text{for all }g,
\]
and $f_s$ is a morphism of varieties because locally $s(q(g))$ and the trivialization depend algebraically on $g$. For $u$ a point of $U$, the classes $[(g,f_s(g))]$ and $[(g\cdot u,f_s(g\cdot u))]$ both equal $s(q(g))$, and by \eqref{eq:pinned-action} we have $[(g\cdot u, u^{-1}\cdot f_s(g))]=[(g,f_s(g))]$; uniqueness of the $L$-coordinate over the fixed point $g\cdot u$ gives
\[
f_s(g\cdot u)=u^{-1}\cdot f_s(g).
\]
Thus $f_s\in\Mor^U(G,L)$.

Conversely, let $f\in\Mor^U(G,L)$. Define $s\colon A\to F_L$ by $s(a):=[(g,f(g))]$ for any point $g\in q^{-1}(a)$. This is well defined: if $g'=g\cdot u$ is another point of the fiber, then
\[
[(g',f(g'))]=[(g\cdot u,\,u^{-1}\cdot f(g))]=[(g,f(g))]
\]
by \eqref{eq:pinned-action}. It is a morphism because on $A_i$ we may take $g=s_i(a)$ with $s_i$ the local sections of Set-up~\ref{setup:associated-bundle}, so that $s(a)=[(s_i(a),f(s_i(a)))]$ is a composition of morphisms. The two constructions are mutually inverse, which proves (1).

(2) A morphism $f$ from $G$ to an affine space $\mathbb A^N$ is given by $N$ global regular functions on $G$. By Lemma~\ref{lem:torsor}(3), $\mathcal O(G)=\mathbb C$, so each coordinate of $f$ is constant, and $f$ is constant. Now let $f\in\Mor^U(G,L)$ be the constant morphism with value $l\in L$. The equivariance condition reads $l=u^{-1}\cdot l$ for all points $u$ of $U$, that is, $l\in L^U$. Combining with (1), we conclude that $H^0(A,F_L)\cong L^U$.

(3) In the local trivializations of Set-up~\ref{setup:associated-bundle}, $\mathcal E$ has transition functions $\rho(u_{ij})$, where $\rho\colon U\to\mathrm{GL}(W)$ is the representation. The bundle $\Sym^m\mathcal E$ then has transition functions $\Sym^m\rho(u_{ij})$, which are exactly the transition functions of $F_{\Sym^mW}$ for the induced representation of $U$ on $\Sym^mW$. Hence $\Sym^m\mathcal E\cong F_{\Sym^mW}$, compatibly with the trivializations.

(4) For $a,b\geq 0$, the multiplication $\Sym^aW\otimes\Sym^bW\to\Sym^{a+b}W$ is $U$-equivariant, hence induces a morphism of associated bundles which, under (3), is the multiplication $\Sym^a\mathcal E\otimes\Sym^b\mathcal E\to\Sym^{a+b}\mathcal E$. Under the identification of (1) and (2), a global section of $\Sym^a\mathcal E$ corresponds to a constant equivariant morphism with value $\alpha\in(\Sym^aW)^U$, and the product of the sections corresponding to $\alpha$ and $\beta$ corresponds to the constant morphism with value $\alpha\beta\in(\Sym^{a+b}W)^U$. Therefore the degreewise isomorphisms assemble into an isomorphism of graded $\mathbb C$-algebras, and the identification with $\mathcal O(V)^U$ is \eqref{eq:invariant-ring}.
\end{proof}

\begin{rem}\label{rem:atiyah}
Proposition~\ref{prop:descent} may be checked against a rank-two shadow from \cite{Ati57}. For an elliptic curve $A$, the universal vector extension is an extension of $A$ by $\mathbb G_a$, and for the two-dimensional representation $W_0$ of $\mathbb G_a$ dual to $(x,y)\mapsto(x,y+tx)$, the associated bundle $G\times^{\mathbb G_a}W_0$ is the nonsplit self-extension of $\mathcal O_A$ by $\mathcal O_A$, the Atiyah bundle $F_2$ (nonsplit, since a splitting would give $\dim_{\mathbb C}H^0(A,G\times^{\mathbb G_a}W_0)=2$, contradicting $\dim_{\mathbb C}W_0^{\mathbb G_a}=1$). Proposition~\ref{prop:descent} then gives $\dim_{\mathbb C}H^0(A,\Sym^mF_2)=\dim_{\mathbb C}(\Sym^mW_0)^{\mathbb G_a}=1$ for all $m\geq 0$, recovering the computation of \cite{Ati57}.
\end{rem}

\section{The projective bundle}\label{sec:projective-bundle}

In this section, we compute the geometry of $X=\mathbb P_A(\mathcal E)$ and verify the generalized klt condition.

\begin{prop}\label{prop:pe}
Let $A$ be an abelian variety of dimension $4$ over $\mathbb C$ and let $\mathcal E$ be a locally free sheaf of rank $16$ on $A$ such that $\mathcal E$ is nef and $\det\mathcal E\cong\mathcal O_A$. Let $X:=\mathbb P_A(\mathcal E)$, $\pi\colon X\to A$, and $\xi:=c_1(\mathcal O_X(1))$ as in Section~\textup{\ref{sec:preliminaries}}. Then:
\begin{enumerate}
    \item $X$ is a smooth projective variety of dimension $19$, and $\xi$ is a nef Cartier divisor class;
    \item $K_X\sim-16\xi$;
    \item there is an isomorphism of graded $\mathbb C$-algebras
    \[
    R(X,\xi)=\bigoplus_{m\geq 0}H^0(X,\mathcal O_X(m))\cong\bigoplus_{m\geq 0}H^0(A,\Sym^m\mathcal E).
    \]
\end{enumerate}
\end{prop}

\begin{proof}
(1) Since $\mathcal E$ is locally free of rank $16$, the morphism $\pi$ is a Zariski-locally trivial $\mathbb P^{15}$-bundle, so $X$ is smooth over $A$ of relative dimension $15$; as $A$ is smooth projective of dimension $4$ and $\pi$ is projective, $X$ is smooth projective of dimension $19$. The class $\xi$ is Cartier because $\mathcal O_X(1)$ is a line bundle, and it is nef because $\mathcal E$ is nef, by the definition in Section~\ref{sec:preliminaries}.

(2) Taking determinants in the relative Euler sequence \eqref{eq:euler} gives \eqref{eq:relative-canonical} with $r=16$:
\[
\omega_{X/A}\cong\det\bigl((\pi^*\mathcal E)\otimes\mathcal O_X(-1)\bigr)\cong\pi^*(\det\mathcal E)\otimes\mathcal O_X(-16)\cong\mathcal O_X(-16),
\]
where we used $\det(F\otimes L)\cong\det F\otimes L^{\otimes 16}$ for a rank-$16$ bundle $F$ and a line bundle $L$, together with $\det\mathcal E\cong\mathcal O_A$. Since $A$ is an abelian variety, $\Omega^1_A$ is trivialized by translation-invariant one-forms, so $\omega_A\cong\mathcal O_A$. As $\pi$ is smooth,
\[
\omega_X\cong\omega_{X/A}\otimes\pi^*\omega_A\cong\mathcal O_X(-16).
\]
Equivalently, $K_X\sim-16\xi$.

(3) By \eqref{eq:pushforward} and the projection to global sections, $H^0(X,\mathcal O_X(m))=H^0(A,\pi_*\mathcal O_X(m))=H^0(A,\Sym^m\mathcal E)$ for every $m\geq 0$. The multiplication $\mathcal O_X(a)\otimes\mathcal O_X(b)\to\mathcal O_X(a+b)$ pushes forward to the multiplication $\Sym^a\mathcal E\otimes\Sym^b\mathcal E\to\Sym^{a+b}\mathcal E$ of the graded $\mathcal O_A$-algebra $\Sym\mathcal E$, so the identifications are compatible with products.
\end{proof}

\begin{prop}\label{prop:gklt}
Notation and hypotheses as in Proposition~\textup{\ref{prop:pe}}. Let $B:=0$ and let $\Mm:=\overline{17\xi}$ be the closure of the nef Cartier divisor $17\xi$, as in Definition~\textup{\ref{defn:bdivisor}}. Then $(X,B,\Mm)$ is a projective klt generalized pair with
\[
K_X+B+\Mm_X\sim\xi.
\]
\end{prop}

\begin{proof}
The variety $X$ is smooth projective, hence normal, and $B=0$ is effective. Since $\xi$ is nef Cartier, $17\xi$ is nef Cartier, so $\Mm=\overline{17\xi}$ is a b-nef b-$\mathbb Q$-Cartier b-divisor which descends on $X$, with trace $\Mm_X=17\xi$; b-nefness holds because pullbacks of nef Cartier divisors under projective morphisms are nef. By Proposition~\ref{prop:pe}(2),
\[
K_X+B+\Mm_X=K_X+17\xi\sim-16\xi+17\xi=\xi,
\]
which is Cartier; in particular $K_X+B+\Mm_X$ is $\mathbb Q$-Cartier, and $(X,B,\Mm)$ is a generalized pair.

For the generalized klt condition, let $g\colon Y\to X$ be any log resolution on which $\Mm$ descends; since $X$ is smooth and $B=0$, $Y$ is smooth and $g$ is a projective birational morphism. By the definition of the closure, $\Mm_Y=g^*\Mm_X$, so writing $K_Y+B_Y+\Mm_Y=g^*(K_X+B+\Mm_X)$ and cancelling $\Mm_Y=g^*\Mm_X$ we obtain
\[
B_Y=-(K_Y-g^*K_X).
\]
Since $Y$ and $X$ are smooth and $g$ is birational, $K_Y-g^*K_X$ is effective: it is the divisor of the Jacobian determinant of $g$, a regular section of $\omega_Y\otimes g^*\omega_X^{-1}$ which is nonzero on the locus where $g$ is an isomorphism. Hence every coefficient of $B_Y$ is $\leq 0<1$, and $(X,B,\Mm)$ is generalized klt.
\end{proof}

\section{Proof of the main theorem}\label{sec:proof}

In this section, we prove Theorem~\ref{thm:main}, and then compute the numerical and Iitaka dimensions of the example. We first fix the data of the example.

\begin{setup}\label{setup:example}
Let $A$ be an abelian variety of dimension $4$ over $\mathbb C$, for example the product of four elliptic curves. Let $U=H^1(A,\mathcal O_A)^\vee$, let $q\colon G=E(A)\to A$ be the universal vector extension, and fix an isomorphism $U\cong(\mathbb G_a)^4$ (Lemma~\ref{lem:torsor}(1)). Let $W$ be Totaro's $16$-dimensional representation \eqref{eq:dual-action}, regarded as a representation of $U$ through this isomorphism, and let $\mathcal E=G\times^UW$ be the associated bundle of Set-up~\ref{setup:associated-bundle}. Let $X:=\mathbb P_A(\mathcal E)$ with projection $\pi\colon X\to A$ and $\xi:=c_1(\mathcal O_X(1))$, and set
\[
B:=0,\qquad \Mm:=\overline{17\xi},\qquad S:=\mathcal O(V)^U=\bigoplus_{m\geq 0}(\Sym^mW)^U.
\]
\end{setup}

\begin{lem}\label{lem:example-ring}
In Set-up~\textup{\ref{setup:example}}, the triple $(X,B,\Mm)$ is a projective klt generalized pair such that $X$ is smooth projective of dimension $19$, $\Mm$ descends on $X$ with $\Mm_X=17\xi$ nef Cartier, and $K_X+B+\Mm_X\sim\xi$. Moreover, there are isomorphisms of graded $\mathbb C$-algebras
\[
R(X,K_X+B+\Mm_X)\cong R(X,\xi)\cong S.
\]
\end{lem}

\begin{proof}
By Proposition~\ref{prop:bundle}, $\mathcal E$ is a nef rank-$16$ bundle with $\det\mathcal E\cong\mathcal O_A$, so Propositions~\ref{prop:pe} and~\ref{prop:gklt} apply: $(X,B,\Mm)$ is a projective klt generalized pair, $X$ is smooth projective of dimension $19$, $\xi$ and hence $\Mm_X=17\xi$ are nef Cartier, $\Mm=\overline{17\xi}$ descends on $X$, and $K_X+B+\Mm_X\sim\xi$.

Since $K_X+B+\Mm_X\sim\xi$, Lemma~\ref{lem:linear-equivalence} gives $R(X,K_X+B+\Mm_X)\cong R(X,\xi)$ as graded $\mathbb C$-algebras. By Proposition~\ref{prop:pe}(3) and Proposition~\ref{prop:descent}(4),
\[
R(X,\xi)\cong\bigoplus_{m\geq 0}H^0(A,\Sym^m\mathcal E)\cong\bigoplus_{m\geq 0}(\Sym^mW)^U=S. \qedhere
\]
\end{proof}

\begin{proof}[Proof of Theorem~\textup{\ref{thm:main}}(1)(2)]
Adopt Set-up~\ref{setup:example}. Part (1) is contained in Lemma~\ref{lem:example-ring}. For part (2), Lemma~\ref{lem:example-ring} identifies $R(X,K_X+B+\Mm_X)$ with $S=\mathcal O(V)^U$ as graded $\mathbb C$-algebras, and $\mathcal O(V)^U$ is not finitely generated over $\mathbb C$ by Theorem~\ref{thm:totaro}. Since finite generation of a graded $\mathbb C$-algebra is invariant under isomorphism, $R(X,K_X+B+\Mm_X)$ is not a finitely generated $\mathbb C$-algebra.
\end{proof}

It remains to prove part (3) of Theorem~\ref{thm:main}. Recall that for a nef Cartier divisor class $\xi$ on a smooth projective variety $X$ of dimension $n$, the \emph{numerical dimension} is
\[
\nu(X,\xi):=\max\{k\in\{0,1,\dots,n\}\mid \xi^k\not\equiv 0\},
\]
where $\equiv$ denotes numerical equivalence of cycle classes, and $\xi$ is big if and only if $\nu(X,\xi)=n$ \cite[Theorem~2.2.16]{Laz04}. We say that a nef Cartier divisor class $\xi$ is \emph{abundant} if $\kappa(X,\xi)=\nu(X,\xi)$, where $\kappa$ denotes the Iitaka dimension recalled before Lemma~\ref{lem:trdeg} below.

\begin{prop}\label{prop:nu}
In Set-up~\textup{\ref{setup:example}}, $\nu(X,\xi)=15$. In particular, $\xi$ is not big.
\end{prop}

\begin{proof}
By the Grothendieck relation defining Chern classes \cite[Appendix~A, Section~3]{Har77},
\[
\sum_{i=0}^{16}(-1)^i\,\pi^*c_i(\mathcal E)\cdot\xi^{16-i}=0
\]
in the Chow ring of $X$. By Proposition~\ref{prop:bundle}(1), $c_i(\mathcal E)=0$ for $i\geq 1$, so the relation reads $\xi^{16}=0$. Hence $\xi^k=0$ for all $k\geq 16$, and $\nu(X,\xi)\leq 15$.

For the reverse inequality, let $H$ be an ample divisor on $A$. The restriction of $\mathcal O_X(1)$ to any fiber of $\pi$ is $\mathcal O_{\mathbb P^{15}}(1)$, so $\pi_*(\xi^{15})=[A]$, and by the projection formula
\[
\xi^{15}\cdot(\pi^*H)^4=\pi_*(\xi^{15})\cdot H^4=H^4>0.
\]
Thus $\xi^{15}\not\equiv 0$, and $\nu(X,\xi)=15<19=\dim X$, so $\xi$ is not big.
\end{proof}

We compute the Iitaka dimension $\kappa(X,\xi)$, defined as the maximum over all $m\geq 1$ with $H^0(X,\mathcal O_X(m\xi))\neq 0$ of the dimension of the image of the rational map defined by the linear system $|m\xi|$; see \cite[Section~2.1.A]{Laz04}.

We first record the transcendence degree of the section ring. Write
\[
\mathbb C[V]=\mathbb C[x_1,\dots,x_8,y_1,\dots,y_8],
\]
so that $S=\mathbb C[V]^U$ in Set-up~\ref{setup:example}, and write $\mathbb C(V)$ for the fraction field of $\mathbb C[V]$. Let $R\subset\mathbb C^8$ denote the row space of the $4\times 8$ matrix of Theorem~\ref{thm:totaro}, so that the vector $t=(t_1,\dots,t_8)$ associated with $u\in U$ ranges over $R$. For $c=(c_1,\dots,c_8)\in\mathbb C^8$, set
\begin{equation}\label{eq:Fc}
F_c:=\sum_{i=1}^8c_i\,y_i\prod_{j\neq i}x_j\in\mathbb C[V],
\end{equation}
and define the four vectors
\begin{equation}\label{eq:cvectors}
c^{(1)}=(1,-1,-1,1,0,0,0,0),\quad
c^{(2)}=(1,-1,0,0,-1,1,0,0),\quad
c^{(3)}=(1,0,-1,0,-1,0,1,0),
\end{equation}
and $c^{(4)}=(1,-1,-1,1,-1,1,1,-1)$, each of which is orthogonal to all four rows of the matrix of Theorem~\ref{thm:totaro} under the standard bilinear form on $\mathbb C^8$.

\begin{lem}\label{lem:trdeg}
With the notation above:
\begin{enumerate}
    \item $F_c\in S$ for every $c\in R^\perp$, and the twelve invariants
    \[
    x_1,\dots,x_8,\ F_{c^{(1)}},\ F_{c^{(2)}},\ F_{c^{(3)}},\ F_{c^{(4)}}
    \]
    are algebraically independent over $\mathbb C$;
    \item $\Frac(S)=\mathbb C(V)^U$ and $\operatorname{trdeg}_{\mathbb C}\Frac(S)=12$.
\end{enumerate}
\end{lem}

\begin{proof}
\textit{Step 1: $\Frac(S)=\mathbb C(V)^U$.}
The inclusion $\Frac(S)\subset\mathbb C(V)^U$ holds because $U$ acts on $\mathbb C(V)$ by field automorphisms fixing $S$. Conversely, let $h\in\mathbb C(V)^U$ and write $h=f/g$ with $f,g\in\mathbb C[V]$ coprime. For a point $u$ of $U$, invariance gives $(u\cdot f)g=f(u\cdot g)$. Since $u$ acts by an algebra automorphism, $u\cdot f$ and $u\cdot g$ are again coprime, so $u\cdot f=c(u)f$ and $u\cdot g=c(u)g$ for some $c(u)\in\mathbb C^*$. The map $u\mapsto c(u)$ is a morphism $U\to\mathbb G_m$: writing $u\cdot f$ in coordinates, $c(u)$ is a ratio of matrix coefficients of the action, which are polynomial in $u$. Moreover $c(0)=1$. Since $U\cong\mathbb A^4$ as a variety, every invertible regular function on $U$ is a nonzero constant, so $c\equiv 1$. Thus $f,g\in S$ and $h\in\Frac(S)$.

\medskip

\noindent\textit{Step 2: part (1), and $\operatorname{trdeg}_{\mathbb C}\mathbb C(V)^U\geq 12$.}
Let $c\in R^\perp$. By \eqref{eq:dual-action} and \eqref{eq:Fc},
\[
u\cdot F_c=\sum_{i=1}^8c_i\,(y_i-t_ix_i)\prod_{j\neq i}x_j=F_c-\Bigl(\sum_{i=1}^8c_it_i\Bigr)x_1\cdots x_8=F_c,
\]
since $t\in R$ and $c\in R^\perp$; so $F_c\in S$. A direct check shows that $c^{(1)},\dots,c^{(4)}$ of \eqref{eq:cvectors} lie in $R^\perp$ and are linearly independent, as the submatrix formed by their coordinates $4,6,7,8$ has rank $4$. To see that the twelve invariants of part (1) are algebraically independent over $\mathbb C$, it suffices to check that their Jacobian matrix with respect to $(x_1,\dots,x_8,y_1,\dots,y_8)$ has rank $12$ at some point. At a point where $x_1=\dots=x_8=1$, the Jacobian contains the identity block $\partial x_i/\partial x_j=\delta_{ij}$ and the block $\partial F_{c^{(k)}}/\partial y_i=c^{(k)}_i$, which has rank $4$; hence the rank is $12$. This proves part (1), and $\operatorname{trdeg}_{\mathbb C}\mathbb C(V)^U\geq 12$ follows from Step 1.

\medskip

\noindent\textit{Step 3: $\operatorname{trdeg}_{\mathbb C}\mathbb C(V)^U\leq 12$.}
The action of $U$ on $V$ is generically free: if $u$ fixes a point $v$ with all coordinates $x_i(v)\neq 0$, then $t_ix_i(v)=0$ for all $i$, so $t=0$, and since the four rows of the matrix are linearly independent, $u=0$. Hence the orbit of such a point $v$, being the image of the injective orbit morphism $U\to V$, has dimension $4$.

Suppose that $h_1,\dots,h_{13}\in\mathbb C(V)^U$ were algebraically independent. Let $\Phi\colon V\dashrightarrow\mathbb A^{13}$ be the rational map they define and let $Z$ be the closure of its image, so that $\mathbb C(Z)\supset\mathbb C(h_1,\dots,h_{13})$ and $\dim Z\geq 13$; as $\dim Z=\operatorname{trdeg}_{\mathbb C}\mathbb C(Z)$ and $\mathbb C(Z)$ is generated by the $h_i$, we have $\dim Z=13$. By the fiber dimension theorem, the fiber of $\Phi$ through a general point $v\in V$ has dimension $16-13=3$. But each $h_i$ is constant on $U$-orbits, so this fiber contains $U\cdot v\cap\operatorname{dom}(\Phi)$; for general $v$ the set $\{u\in U\mid u\cdot v\in\operatorname{dom}(\Phi)\}$ is a dense open subset of $U$, so $U\cdot v\cap\operatorname{dom}(\Phi)$ has dimension $4$. This is a contradiction. Hence $\operatorname{trdeg}_{\mathbb C}\mathbb C(V)^U\leq 12$, and combining the three steps completes the proof of part (2).
\end{proof}

\begin{prop}\label{prop:kappa}
In Set-up~\textup{\ref{setup:example}}, $\kappa(X,\xi)=11$.
\end{prop}

\begin{proof}
By Lemma~\ref{lem:example-ring}, there are isomorphisms of graded $\mathbb C$-algebras $R(X,\xi)\cong S=\mathcal O(V)^U$; we henceforth identify the two and write $S_m=(\Sym^mW)^U$ for the degree-$m$ piece, so that $H^0(X,\mathcal O_X(m\xi))\cong S_m$ compatibly with multiplication. Note that $x_1\in S_1$ is a nonzero element of degree one.

Let $Q_0\subset\Frac(S)$ denote the subfield of degree-zero homogeneous fractions, that is, of elements $s/s'$ with $s,s'\in S_m$ for some common $m$ and $s'\neq 0$. Since $x_1\in S_1$, every homogeneous $s\in S_d$ satisfies $s=(s/x_1^d)\,x_1^d$, so $\Frac(S)=Q_0(x_1)$. Moreover $x_1$ is transcendental over $Q_0$: a nontrivial algebraic relation $\sum_ka_kx_1^k=0$ with $a_k\in Q_0$ has homogeneous components of pairwise distinct degrees, so each term $a_kx_1^k$ vanishes separately, and as $S$ is a domain, all $a_k=0$. Hence, by Lemma~\ref{lem:trdeg}(2),
\[
\operatorname{trdeg}_{\mathbb C}Q_0=\operatorname{trdeg}_{\mathbb C}\Frac(S)-1=11.
\]

We first prove $\kappa(X,\xi)\leq 11$. For any $m\geq 1$ with $S_m\neq 0$, the closure $Z_m$ of the image of the rational map $\varphi_{|m\xi|}$ defined by $|m\xi|$ has function field generated by ratios of elements of $S_m$, hence $\mathbb C(Z_m)\subset Q_0$ and
\[
\dim Z_m=\operatorname{trdeg}_{\mathbb C}\mathbb C(Z_m)\leq\operatorname{trdeg}_{\mathbb C}Q_0=11.
\]

We next prove $\kappa(X,\xi)\geq 11$. Take $m=8$. By Lemma~\ref{lem:trdeg}(1), the degree-$8$ piece $S_8$ contains the twelve elements
\[
x_1^8,\ x_1^7x_2,\ \dots,\ x_1^7x_8,\ F_{c^{(1)}},\ \dots,\ F_{c^{(4)}},
\]
with $F_{c^{(k)}}$ as in \eqref{eq:Fc} and \eqref{eq:cvectors}. The function field of $Z_8$ contains the eleven ratios
\[
\frac{x_1^7x_i}{x_1^8}=\frac{x_i}{x_1}\ (2\leq i\leq 8),\qquad
\frac{F_{c^{(k)}}}{x_1^8}\ (1\leq k\leq 4).
\]
Let $K_0\subset Q_0$ be the subfield they generate. Then $K_0(x_1)$ contains $x_1,\dots,x_8$ and $F_{c^{(1)}},\dots,F_{c^{(4)}}$, which are algebraically independent by Lemma~\ref{lem:trdeg}(1), so $\operatorname{trdeg}_{\mathbb C}K_0(x_1)\geq 12$, whence $\operatorname{trdeg}_{\mathbb C}K_0\geq 11$. Therefore
\[
\dim Z_8=\operatorname{trdeg}_{\mathbb C}\mathbb C(Z_8)\geq\operatorname{trdeg}_{\mathbb C}K_0\geq 11.
\]
Combining the two inequalities, we conclude that $\kappa(X,\xi)=11$.
\end{proof}

\begin{rem}\label{rem:rosenlicht}
The transcendence degree $\trdeg_{\mathbb C}\Frac(S)=12$ of Lemma~\ref{lem:trdeg}(2) can also be obtained directly from Rosenlicht's theorem \cite{Ros56} (see also \cite[Theorem 1.1, \S 7]{BGR17}), without exhibiting an explicit transcendence basis. Let $U\cong(\mathbb G_a)^4$ act on $V=\mathbb A^{16}$ by Totaro's representation, and let $S:=\mathcal O(V)^U$ be the ring of invariants. The action is generically free: a point with all $x_i\neq 0$ has trivial stabilizer, since $u\cdot v=v$ forces $t(u)=0$, hence $u=0$ (the coefficient matrix of Theorem~\ref{thm:totaro} has rank $4$). Since $U$ is connected, Rosenlicht's theorem gives
\[
\trdeg_{\mathbb C}\mathbb C(V)^U=\dim V-\dim U=16-4=12.
\]
Together with the equality $\Frac(S)=\mathbb C(V)^U$ of Lemma~\ref{lem:trdeg}(2), this yields $\trdeg_{\mathbb C}\Frac(S)=12$.

With this single number, Proposition~\ref{prop:kappa} can be derived without the twelve explicit invariants of Lemma~\ref{lem:trdeg}(1) and without the analysis of $Z_8$. Indeed, since $x_1\in S_1$ is a nonconstant element of degree one, the degree-zero subfield $Q_0\subset\Frac(S)$ satisfies $\Frac(S)=Q_0(x_1)$, with $x_1$ transcendental over $Q_0$ (precisely as in the proof of Proposition~\ref{prop:kappa}), so
\[
\trdeg_{\mathbb C}Q_0=\trdeg_{\mathbb C}\Frac(S)-1=11.
\]
By the classical $D$-dimension identity of Iitaka's theory (see Section~2.1 of \cite{Laz04}),
\[
\kappa(X,\xi)=11.
\]
Thus, once Rosenlicht's theorem is accepted, Proposition~\ref{prop:kappa} follows from $16-4=12$ and $12-1=11$.
\end{rem}

\begin{proof}[Proof of Theorem~\textup{\ref{thm:main}}(3)]
By Lemma~\ref{lem:example-ring}, $K_X+B+\Mm_X\sim\xi$ with $\xi$ nef Cartier. By Propositions~\ref{prop:nu} and~\ref{prop:kappa}, $\kappa(X,\xi)=11<15=\nu(X,\xi)$, so $\xi$ is not big and not abundant.
\end{proof}

\begin{rem}\label{rem:sharpness}
Part (3) shows that the failure of finite generation in Theorem~\ref{thm:main} does not come from scarcity of sections: the ring $R(X,K_X+B+\Mm_X)$ has transcendence degree $12$ over $\mathbb C$. It also locates the example squarely outside the reach of the known positive results recalled in Section~\ref{sec:introduction}: $K_X+B+\Mm_X$ is neither big nor abundant.
\end{rem}

\end{document}